\documentclass[11pt]{article}

\usepackage{latexsym,verbatim,ifthen,graphicx,mathrsfs}

\usepackage[english]{babel}
\usepackage[utf8]{inputenc} 		
\usepackage{amssymb}
\usepackage{amsmath}
\usepackage{amsthm}			
\usepackage{tikz-cd}				
\usetikzlibrary{matrix,arrows,decorations.pathmorphing}
\usepackage[all]{xy}				
\usepackage{hyperref}  			
\usepackage{anysize}						
\marginsize{2.5cm}{2.5cm}{2cm}{2cm}
\usepackage{booktabs}			
\usepackage{enumerate}			
\usepackage{multicol}

\usepackage{mathrsfs}									
\usepackage{fourier}				
\usepackage[Sonny]{fncychap}		
\usepackage{fancyhdr}

\newtheorem{theorem}{Theorem}[section]
\newtheorem{lemma}[theorem]{Lemma}
\newtheorem{proposition}[theorem]{Proposition}
\newtheorem{corollary}[theorem]{Corollary}

\theoremstyle{definition}
\newtheorem{remark}[theorem]{Remark}

\newtheorem{definition}[theorem]{Definition}
\newtheorem{conjecture}[theorem]{Conjecture}

\newtheorem{thm}{Theorem}

\DeclareMathOperator{\Ker}{Ker}

\DeclareMathOperator{\sgn}{sgn}

\DeclareMathOperator{\id}{id}

\newcommand{\Z}{{\mathbb{Z}}}
\newcommand{\Q}{{\mathbb{Q}}}

\newcommand{\C}{{\mathbb{C}}}
\newcommand{\K}{{\mathbb{K}}}	
\newcommand{\Lie}{{\mathbb{L}}}

\definecolor{red}{rgb}{1,0.1,0.1}
\definecolor{blue}{rgb}{0.1,0.1,1}

\begin{document}

\title{The Milnor-Moore theorem for $L_\infty$ algebras in rational homotopy theory}
\author{José Manuel Moreno-Fernández{\let\thefootnote\relax\footnote{{The author has been partially supported by the MINECO grant MTM2016-78647-P and by a Postdoctoral Fellowship of the Max Planck Society. \vskip 1pt 2010 Mathematics subject
				classification: 55P62, 16S30, 17B55, 18G55, 16E45, 55S30, 55Q15.\vskip
				1pt
				Key words and phrases: Universal enveloping algebra. Rational homotopy theory. $A_\infty$-algebra. $L_\infty$-algebra. Loop space homology. Higher Whitehead products. Massey-Pontryagin products.}}}}
\date{}
\maketitle
\abstract{We give a construction of the universal enveloping $A_\infty$ algebra of a given $L_\infty$ algebra, alternative to the already existing versions. As applications, we derive a higher homotopy algebras version of the classical Milnor-Moore theorem, proposing a new $A_\infty$ model for simply connected rational homotopy types, and uncovering a relationship between the higher order rational Whitehead products in homotopy groups and the Pontryagin-Massey products in the rational loop space homology algebra.}

\section{Introduction} 

The main goal of this paper is to construct a universal enveloping $A_\infty$ algebra for a given $L_\infty$ algebra, alternative to the already existing versions \cite{Lad95,Bar08}, and to study the consequences of such an structure in rational homotopy theory.

Let $L$ be an $L_\infty$ algebra. In Def. \ref{DefUnive}, we introduce the universal enveloping $A_\infty$ algebra $U_t(L)$. It is isomorphic to the free symmetric algebra $\Lambda L$ on $L$ as a graded vector space, and arises from a transfer process. For dg Lie algebras, $U_t(L)$ coincides with the classical dg associative envelope $UL$. To motivate the definition of $U_t$, we first prove the following result (Thm. \ref{Uno}\textit{$(i)$}).

\begin{thm} {\em 
	Let $L$ and $UL$ be a dg Lie algebra and its classical universal enveloping dg associative algebra, respectively. Fix a contraction from $L$ onto $H=H_*(L)$, and denote by $\{\ell_n\}$ the induced $L_\infty$ structure on $H$. Then, there is an explicit contraction from $UL$ onto $\Lambda H$, so that denoting by $\{m_n\}$ the induced $A_\infty$ algebra structure on $\Lambda H$, the antisymmetrization $\left\{m_n^\mathcal L\right\}$ of $\{m_n\}$ fits into a strict $L_\infty$ embedding  $$\imath: \left(H, \left\{\ell_n \right\}\right) \hookrightarrow \left(\Lambda H, \{m_n^\mathcal L\}\right).$$ That is, for every $x_i \in H,$
		\begin{equation*}
		\imath\ell_n(x_1,...,x_n) =  \sum_{ \sigma\in S_n}\chi(\sigma)\ m_n\left(x_{\sigma(1)}\otimes \cdots \otimes x_{\sigma(n)}\right)=m_n^\mathcal L (x_1,...,x_n).
		\end{equation*}}
\end{thm}

The result above covers the case in which $L$ is minimal, since any such can be obtained as a contraction of the dg Lie algebra $\mathcal L \mathcal C(L)$. In general, $U_t(L)$ is defined as $\Lambda L$ together with an $A_\infty$ structure inherited from a contraction from $\Omega \mathcal C(L)$ onto $\Lambda L$. Here,  $\mathcal C$ are the Quillen chains, $\Omega $ the cobar construction, and $\mathcal L$ Quillen's Lie functor. See Section \ref{Seccion2} for details.

 The original motivation for introducing the envelope we present was for extending the classical Milnor-Moore theorem (\cite{Mil65}) to $L_\infty$ algebras in the rational setting. This is Thm. \ref{MMInfinito}.

\begin{thm}\label{MMIntro}{\em Let $X$ be a simply connected CW-complex. Endow $\pi_*\left(\Omega X\right)\otimes \Q$ with an $L_\infty$ structure $\{\ell_n\}$ representing the rational homotopy type of $X$ for which $\ell_1=0$ and $\ell_2=[-,-]$ is the Samelson bracket. Then, there exists an $A_\infty$ algebra structure $\{m_n\}$ on the loop space homology algebra $H_*\left(\Omega X;\Q\right)$ for which $m_1=0, m_2$ is the Pontryagin product, and such that the rational Hurewicz morphism 
		\begin{equation*}
		h : \pi_*\left(\Omega X\right) \otimes \Q \hookrightarrow H_*(\Omega X;\Q)=U_t\left(\pi_*\left(\Omega X\right) \otimes \Q\right)
		\end{equation*} is a strict $L_\infty$ embedding. Therefore, the $L_\infty$ structure on the rational homotopy Lie algebra is the antisymmetrized of the $A_\infty$ structure on $H_*(\Omega X;\Q)$: $$\ell_n(x_1,...,x_n)=\sum_{\sigma\in S_n} \chi(\sigma) m_n\left(x_{\sigma(1)},...,x_{\sigma(n)}\right).$$}
\end{thm}

Thm. \ref{MMIntro} produces a new $A_\infty$ model for simply connected rational homotopy types, with underlying Hopf algebra $H_*(\Omega X;\Q)$. For finite type rational spaces, this enveloping $A_\infty$ algebra model  can be understood as an Eckmann-Hilton or Koszul dual to Kadeishvili's $C_\infty$ algebra model \cite{Kad09}, the latter starting from cohomology instead of homotopy. We explain in Section \ref{Ejemplos} how to explicitly extract the Quillen and Sullivan models from such an enveloping $A_\infty$ model. We also uncover an interesting relationship between the  higher order rational Whitehead products on $\pi_*\left(\Omega X\right)\otimes \Q$ and the higher order Pontryagin-Massey products of $H_*\left(\Omega X;\Q\right)$ of simply connected spaces: the former are antisymmetrizations of the latter, whenever these are defined. This is Thm. \ref{HigherWhiteheadHigherPontryagin}. In it, $h$ is the rational Hurewicz morphism.

\begin{thm} \em Let $x_1,...,x_n \in \pi_*\left( \Omega X\right)\otimes \Q$, and denote by $y_k=h\left(x_k\right)\in H_*\left(\Omega X;\Q\right)$ the corresponding spherical classes. Assume that the higher Whitehead product set $\left[x_1,...,x_n\right]_W$ and the higher Massey-Pontryagin products sets $\left\langle y_{\sigma(1)},...,y_{\sigma(n)} \right\rangle$ for every permutation $\sigma \in S_n$ are defined. If the $A_\infty$ algebra structure $\{m_i\}$ on $H_*\left(\Omega X;\Q\right)$ provided by Theorem \ref{MMIntro} has vanishing $m_k$ for $k\leq n-2$, then $x=\varepsilon\ell_n\left(x_1,...,x_n\right)\in \left[x_1,...,x_n\right]_W,$ and satisfies:
	\begin{equation*}
	h(x) \in \sum_{\sigma \in S_n} \chi(\sigma) \left\langle y_{\sigma(1)},...,y_{\sigma(n)}\right\rangle.
	\end{equation*} Here, $\varepsilon$ is the parity of $\sum_{j=1}^{n-1} |x_j|(k-j)$. If moreover the higher products are all uniquely defined, then the above containment is an equality of elements.
\end{thm}

The Massey-Pontryagin products should not be confused with the classical Massey products, see Section \ref{SeccionHigherPontryagin} for details. We study the homotopical properties of the envelope $U_t$, and we compare it to other alternatives in the literature in Section \ref{comparacion}. These alternative constructions have been developed by Lada and Markl \cite{Lad95} and by Baranovsky \cite{Bar08}.  See Prop. \ref{SonHomotopas} for a recollection of the statements. In particular, the classical identity $UH=HU$, asserting that taking homology and universal enveloping algebra commute, holds only up to homotopy for any sort of enveloping $A_\infty$ algebra, and $U_t$ is quasi-isomorphic to Baranovsky's construction.\\

\noindent {\bf Acknowledgements:} The author is very grateful to Martin Markl, Aniceto Murillo, Peter Teichner and Felix Wierstra for useful feedback on this project, and also to the Max Planck Institute for Mathematics in Bonn for its hospitality and financial support.

\subsection{Background and notation} In this paper, graded objects are always taken over $\Z$, with homological grading (differentials lower the degree by 1). The degree of an element $x$ is denoted by $|x|$, and all algebraic structures are considered over a characteristic zero field. \\ 

An \emph{A$_\infty$ algebra} is a graded vector space $A=\{A_n\}_{n\in \Z}$ together with linear maps $m_k:A^{\otimes k}\to A$ of degree $k-2$, for $k\geq 1$, satisfying the \emph{Stasheff identities} for every $i\geq1$:
\begin{equation*}\label{highermu}
	\sum_{k=1}^i\, \sum_{n=0}^{i-k}(-1)^{k+n+kn}m_{i-k+1}({\rm id}^{\otimes n}\otimes m_k\otimes{\rm id}^{\otimes i-k-n})
	=0.
\end{equation*}
A {\em  differential graded algebra} (DGA), is an $A_\infty$ algebra for which $m_k=0$ for $k\ge 3$. An $A_\infty$ algebra is \emph{minimal} if $m_1=0$. An \emph{$A_\infty$ morphism} $f\colon A\to B$ is a family of linear maps $f_k\colon A^{\otimes k}\to B$ of degree $k-1$ such that the following equation holds for every $i\geq1$:
\begin{eqnarray*}\label{highermap}
\sum_{\substack{i=r+s+t \\ s\geq 1 \\ r,t\geq 0}} (-1)^{r+st}f_{r+1+t} \left(\id^{\otimes r}\otimes m_s\otimes \id^{\otimes t}\right)= \sum_{\substack{1\leq r \leq i \\ i=i_1+\cdots+i_r}} (-1)^s m_r \left(f_{i_1}\otimes \cdots \otimes f_{i_r}\right)
\end{eqnarray*} being $s=\sum_{\ell=1}^{r-1}\ell(i_{r-\ell}-1).$ Such an $f$ is an \emph{$A_\infty$ quasi-isomorphism} if $f_1\colon(A,m_1)\to (A',m_1')$ is a quasi-isomorphism of complexes. The \emph{bar construction} $BA$ of an $A_\infty$ algebra $A$ is the differential graded coalgebra (DGC, henceforth) $$BA=\left(T\left(sA\right),\delta\right),$$ where $T\left(sA\right)$ is the tensor coalgebra on the suspension $sA$ of $A$ (i.e., $\left(sA\right)_p=A_{p-1}$), and $\delta=\sum_{k\geq 1} \delta_k$ is the codifferential such that 
\begin{equation*}
	\delta_k [sx_1\mid \cdots \mid sx_p] = \sum_{i=0}^{p-k+1} \varepsilon_i [sx_1\mid \cdots \mid sx_i \mid sm_{k+1}\left(x_{i+1},...,x_{i+k+1}\right)\mid \cdots \mid sx_p],
\end{equation*} where $\varepsilon_i$ is the parity of $1+\sum_{j=1}^{i} |sx_j| + \sum_{l=1}^{k+1} (k+1-j)|sx_{i+l}|.$ The bar construction turns $A_\infty$ morphisms $A\to C$ into DGC morphisms $BA\to BC$, and preserves quasi-isomorphisms (\cite{Lod12}). The \emph{cobar construction} $\Omega C$ of a coaugmented DGC $C$ is the augmented DGA $$\Omega C=\left(T\left(s^{-1}\overline C\right),d\right),$$ where $T\left(s^{-1}\overline C\right)$ is the tensor algebra on the desuspension $s^{-1}\overline C$ of the cokernel $\overline C= \operatorname{coKer}\left(\K\to C\right)$ of the coaugmentation $\K\to C$ (i.e., $(s^{-1}\overline C)_p=\overline C_{p+1}$), and $d=d_1 + d_2$ is the differential determined by
\begin{equation*}
	d_1\left(s^{-1}x\right) = -s^{-1}\delta x, \qquad d_2\left(s^{-1}x\right)= \sum_{i} (-1)^{|x_i|} s^{-1}x_i\otimes s^{-1}y_i,
\end{equation*} where $\delta$ is the codifferential of $C$ and $\sum_i x_i\otimes y_i = \Delta(x) - \left(1\otimes x+x\otimes 1\right)$ is the reduced comultiplication of $x$. The cobar construction extends to $A_\infty$ coalgebras, but we are not in the need of such a generality in this paper.

An \emph{$L_\infty$ algebra} is a graded vector space $L=\left\{L_n\right\}_{n\in \Z}$ together with skew-symmetric linear maps $\ell_k:L^{\otimes k}\to L$ of degree $k-2$, for $k\geq 1$, satisfying the \emph{generalized Jacobi identities} for every $n\geq 1$: $$\sum_{i+j=n+1} \sum_{\sigma \in S(i,n-i) } \varepsilon(\sigma) \sgn(\sigma)(-1)^{i(j-1)} \ell_{j}\left(\ell_i\left(x_{\sigma(1)},...,x_{\sigma(i)}\right),x_{\sigma(i+1)},...,x_{\sigma(n)}\right) =0.$$

  Here, $S(i,n-i)$ are the $(i,n-i)$ shuffles, given by those permutations $\sigma$ of $n$ elements such that $$\sigma(1)<\cdots < \sigma(i) \quad \textrm{ and } \quad \sigma(i+1)<\cdots < \sigma(n);$$ and $\varepsilon(\sigma), \sgn(\sigma)$ stand for the Koszul sign and the signature associated to $\sigma$, respectively.  A \emph{differential graded Lie algebra} (DGL) is an $L_\infty$ algebra $L$ for which $\ell_k=0$ for $k\geq 3$.

An $L_\infty$ algebra is \emph{minimal} if $\ell_1=0.$ An \emph{$L_\infty$ morphism} $f:L \to L'$ is a family of skew-symmetric linear maps $\left\{ f_n: L^{\otimes n} \to L'\right\}$ of degree $n-1$ such that the following equation is satisfied for every $n\geq 1$:

\begin{equation*}\begin{split} \sum_{i+j=n+1} \ \ \sum_{\sigma \in S(i,n-i)} \varepsilon(\sigma)\sgn(\sigma)(-1)^{i(j-1)} f_j\left(\ell_i\left(x_{\sigma(1)},...,x_{\sigma(i)}\right),x_{\sigma(i+1)},...,x_{\sigma(n)}\right) = \\ \sum_{\substack{k\geq 1 \\ i_1+\cdots +i_k =n \\ \tau \in S(i_1,...,i_k)}} \varepsilon(\sigma)\sgn(\sigma)\varepsilon_k \ell_k'\left(f_{i_1}\otimes \cdots \otimes f_{i_k}\right) \left(x_{\tau(1)}\otimes \cdots \otimes x_{\tau(n)}\right),\end{split}\label{EcuLinf}\end{equation*} with $\varepsilon_k$ being the parity of $\sum_{l=1}^{k-1}(k-l)(i_l-1)$. Such an $f$ is an \emph{$L_\infty$ quasi-isomorphism} if $f_1\colon(L,\ell_1)\to (L',\ell_1')$ is a quasi-isomorphism of complexes. The \emph{Quillen chains} $\mathcal C(L)$ of an $L_\infty$ algebra is the equivalent cocommutative DGC (CDGC, henceforth) $$\mathcal C(L)=\left(\Lambda sL,\delta\right),$$ where $\Lambda sL$ is the cofree conilpotent cocommutative graded coalgebra on the suspension $sL$ of $L$, and $\delta=\sum_{k\geq 1} \delta_k$ is the codifferential whose correstrictions are determined by the $L_\infty$ structure maps, i.e., 
\begin{equation}
\delta_k\left(sx_1\wedge...\wedge sx_p\right)=\sum_{i_1<\dots<i_k}\varepsilon\, s\ell_k\left(x_{i_1},...,x_{i_k}\right)\wedge sx_1\wedge...\widehat{ sx}_{i_1}...\widehat{sx}_{i_k}...\wedge sx_p.
\end{equation} The sign $\varepsilon$ is determined by the Koszul sign rule.\\

A morphism $f=\{f_k\}$ of $A_\infty$ or $L_\infty$ algebras is \emph{strict} if  $f_k=0$ for all $k\geq 2$.\\

The \emph{Quillen functor} $\mathcal L(C)$ on a coaugmented CDGC $C$ is the DGL $$\mathcal L(C)=\left(\Lie\left(s^{-1}\overline C\right), \partial\right),$$ where $\Lie\left(s^{-1}\overline C\right)$ is the free graded Lie algebra on the desuspension $s^{-1}\overline C$ of the cokernel of the coaugmentation, $\overline C= \operatorname{coKer}\left(\K \to C\right)$, and $\partial=\partial_1 + \partial_2$ is the differential determined by
\begin{equation}\label{FormulaPartial}
	\partial_1 \left(s^{-1}x\right) = - s^{-1}\delta(x), \qquad \partial_2\left(s^{-1}x\right) = \frac{1}{2} \sum_i (-1)^{|x_i|} [s^{-1}x_i,s^{-1}y_i],
\end{equation} where $\delta$ is the codifferential of $C$ and $\sum_i x_i\otimes y_i$ is again the reduced comultiplication of $x$.

There is an \emph{antisymmetrization functor $(-)^{\mathcal L}$} from the category of $A_\infty$ algebras to that of $L_\infty$ algebras which preserves quasi-isomorphisms (\cite{Lad95}). For a given $A_\infty$ algebra $(A,\{m_n\})$, its antisymmetrization $A^{\mathcal L}$ has the same underlying graded vector space and higher brackets $\ell_n$ given by 

		\begin{equation*}
		\ell_n(x_1,...,x_n) =  \sum_{ \sigma\in S_n}\chi(\sigma)\ m_n\left(x_{\sigma(1)}\otimes \cdots \otimes x_{\sigma(n)}\right).
		\end{equation*} Here, $S_n$ is the symmetric group on $n$ letters, and we shorten the notation by $\chi(\sigma)=\varepsilon(\sigma) \sgn(\sigma)$ for $\sigma \in S_n$. We will usually denote the higher brackets $\ell_n$ of $A^\mathcal L$ by $m_n^{\mathcal L}$. \\

A \emph{contraction} of $M$ onto $N$ is a diagram of the form 
\begin{equation*}
\begin{tikzcd}
M \ar[loop left]{l}{K} \ar[shift left]{r}{q} & N, \ar[shift left]{l}{i}
\end{tikzcd}
\end{equation*} where $M$ and $N$ are chain complexes and $q$ and $i$ are chain maps such that $qi=\id_N$ and $iq\simeq id_M$ via a chain homotopy $K$ satisfying $K^2=Ki=qK=0.$ We denote it by $(M,N,i,q,K)$, or simply by $(i,q,K)$.\\

Following \cite[Def. 2.3]{Man10}, a \emph{morphism of contractions} $f: (M,N,i,q,K) \to (A,B,j,p,G)$ is a chain map $f:M\to A$ such that $fK=Gf.$ Denote by $\widehat f: N\to B$ the chain map $\widehat f = pfi.$ Using that $iq\simeq \id_M$, it follows that in presence of a morphism of contractions $f:M\to A$, the interior squares in the following diagram commute:
\begin{equation*}
\begin{tikzcd}
A \ar[loop left]{l}{G} \ar[shift left]{r}{p} 							& B \ar[shift left]{l}{j}\\[0.35cm]
M \ar[loop left]{l}{K} \ar[shift left]{r}{q} \ar{u}{f}		& N. \ar[shift left]{l}{i} \ar[,swap]{u}{\widehat f}
\end{tikzcd}
\end{equation*} That is, $pf=\widehat f q$ and $fi=j\widehat f.$\\

We will be concerned with the following particular instance of the \emph{homotopy transfer theorem} (see \cite{Kad80,Mer99,Man10,Hue11A,Lod12,Ber14A}).

\begin{theorem}\label{HTT} Let $(M,N,i,q,K)$ be a contraction.
	\begin{enumerate}
		\item If $M=\left(A,\left\{\mu_n\right\}  \right)$ is an $A_\infty$ algebra, then there exists an $A_\infty$ algebra structure $\{m_n\}$ on $N$, unique up to isomorphism, and $A_\infty$ algebra quasi-isomorphisms 
		\begin{equation*}
		\begin{tikzcd}
		Q:(A,\{\mu_n\}) \ar[shift left]{r} & (N,\left\{m_n\right\}):I \ar[shift left]{l}
		\end{tikzcd}
		\end{equation*} such that $I_1=i$, $Q_1=q$ and $QI=\id_N$.

		\item If $M=(L,\{\vartheta_n\})$ is an $L_\infty$ algebra, then there exists an $L_\infty$ algebra structure $\{\ell_n\}$ on $N$, unique up to isomorphism, and $L_\infty$ algebra quasi-isomorphisms 
		\begin{equation*}
		\begin{tikzcd}
		Q:(L,\{\vartheta_n\}) \ar[shift left]{r} & (N,\left\{\ell_n\right\}):I \ar[shift left]{l}
		\end{tikzcd}
		\end{equation*} such that $I_1=i$, $Q_1=q$ and $QI=\id_N$. 
	\end{enumerate}
\end{theorem}

The maps involved in the higher structure of Theorem \ref{HTT} can be described in several ways. For the purposes of this paper, we will describe the maps using recursive algebraic formulas. We will consistently use the following convention for the rest of the paper: contractions of an $L_\infty$ algebra will be denoted by $(i,q,K)$, whereas contractions of an $A_\infty$ algebra will be denoted by $(j,p,G)$. The capital letters $I,Q$ or $J,P$ will stand for the corresponding induced infinity quasi-isomorphisms.

The higher multiplications $\left\{m_n\right\}$ on $N$ and the terms $\left\{J_n\right\}$ of the $A_\infty$ quasi-isomorphism $J$ are recursively given as follows. Formally, set $G\lambda_1=-j$, and define $\lambda_n:H^{\otimes n}\to A$ for $n\geq 2$ recursively by 
\begin{equation*}
\lambda_n(x_1,...,x_n) = \sum_{k=2}^{n} m_k \left( \sum_{i_1+ \cdots + i_k = n}(-1)^{\alpha(i_1,...,i_k)} G\lambda_{i_1}\otimes \cdots \otimes G\lambda_{i_k} \right) (x_1\otimes \cdots \otimes x_n).
\end{equation*}
Here, $\alpha(i_1,...,i_k)=\sum_{j<k}i_j(i_k-1)$, see \cite[\S 12]{Ber14A}. Then, 
\begin{equation*}
m_n = p \circ \lambda_n \quad \textrm{ and } \quad J_n = G\circ \lambda_n \quad \textrm{ for all } n\geq 2.
\end{equation*}

Similarly, the higher brackets $\left\{ \ell_n\right\}$ and the Taylor series $\left\{I_n\right\}$ of the $L_\infty$ quasi-isomorphism $I$ are recursively given as follows. Formally, set $K\theta_1=-i$, and define $\theta_n:H^{\otimes n}\to L$ for $n\geq 2$ recursively by
\begin{equation*}
\theta_n\left(x_1,...,x_n\right) = \sum_{k=2}^{n} \ \ \sum_{\substack{i_1+\cdots +i_k=n \\ i_1 \leq \cdots \leq i_k}} \ \  \sum_{\widetilde S(i_1,...,i_k)} (-1)^{\varepsilon_\sigma} \ell_k\left( I_{i_1}\left(x_{\sigma(1)},...,x_{\sigma(i_1)}\right),...,I_{i_k}\left(x_{\sigma(i_{k-1}+1)},...,x_{\sigma(n)}\right)\right).
\end{equation*} In the equation above, $\widetilde S(i_1,...,i_k)$ are the \emph{$(i_1,...,i_k)$-shuffle permutations} of the symmetric group $S_n$, whose elements are those $\sigma\in S_n$ such that $\sigma(1)=1$, and
\begin{equation*}
\sigma(1) < \cdots < \sigma(i_1), \quad \sigma(i_1+1)< \cdots < \sigma(i_2), \quad ... ,\quad  \sigma(i_{k-1}+1)< \cdots < \sigma(n).
\end{equation*} The sign $\varepsilon_\sigma$ is determined by the Koszul convention. Then, 
\begin{equation*}
\ell_n = q \circ \theta_n \quad \textrm{ and } \quad I_n = K\circ \theta_n \quad \textrm{ for all } n\geq 2.
\end{equation*}

\section{The universal enveloping $A_\infty$ algebra as a transfer}\label{Seccion2} 

We produce the universal enveloping $A_\infty$ algebra of a given $L_\infty$ algebra via a transfer process. To do so, we start by showing (Thm. \ref{Uno}) that the classical adjoint pair $$U:DGL \leftrightarrows DGA: (-)^\mathcal L$$ commutes with the transfer of higher structure. See \cite[Chap. 21]{Yve12} for a careful exposition of the adjoint pair above. After the proof of Thm. \ref{Uno}, we explain how to produce such a universal envelope, which turns out to coincide with Baranovsky's construction \cite{Bar08} up to homotopy.

\begin{theorem}\label{Uno}
	Let $L$ and $UL$ be a DGL and its classical universal enveloping DGA, respectively. Fix a contraction from $L$ onto $H=H_*(L)$, and denote by $\{\ell_n\}$ the induced $L_\infty$ structure on $H$. Then, there is an explicit contraction from $UL$ onto $\Lambda H$, so that denoting by $\{m_n\}$ the induced $A_\infty$ algebra structure on $\Lambda H$:
	\begin{enumerate}
		\item[{(i)}]The antisymmetrization $\left\{m_n^\mathcal L\right\}$ of $\{m_n\}$ fits into a strict $L_\infty$ embedding  $$\imath: \left(H, \left\{\ell_n \right\}\right) \hookrightarrow \left(\Lambda H, \{m_n^\mathcal L\}\right),$$ that is, for every $x_i \in H,$
		\begin{equation*}
		\ell_n(x_1,...,x_n) =  \sum_{ \sigma\in S_n}\chi(\sigma)\ m_n\left(x_{\sigma(1)}\otimes \cdots \otimes x_{\sigma(n)}\right)=m_n^\mathcal L (x_1,...,x_n).
		\end{equation*}
		\item[{(ii)}]  The $A_\infty$ algebra $\left(\Lambda H, \{m_n^\mathcal L\}\right)$ is isomorphic to Baranovsky's enveloping construction on $\left(H, \left\{\ell_n \right\}\right)$.
	\end{enumerate}
\end{theorem}

The map $\imath:H\hookrightarrow \Lambda H$ above is an $L_\infty$ version of a PBW map $L\hookrightarrow UL.$ The proof of Thm. \ref{Uno} relies in the following lemma, which is elementary but interesting in itself. It will be relevant for the enveloping $A_\infty$ algebra as a transferred structure (Def. \ref{DefUnive}).

\begin{lemma}\label{Principal} Let $(A,\{\mu_n\})$ and $(L,\{\vartheta_n\})$ be an $A_\infty$ and an $L_\infty$ algebra, and assume that there are contractions of $A$ and of $L$ onto complexes $(M_A,d)$ and $(M_L,\partial)$, respectively:
	\begin{equation*}
	\begin{tikzcd}
	A \ar[loop left]{l}{G} \ar[shift left]{r}{p} & M_A \ar[shift left]{l}{j} & \quad & L \ar[loop left]{l}{K} \ar[shift left]{r}{q} & M_L. \ar[shift left]{l}{i}
	\end{tikzcd}
	\end{equation*} 
		\item If there is a morphism of contractions $f:L\to A$ which is a strict $L_\infty$ morphism for the antisymmetrization of the $A_\infty$ algebra structure $\{\mu_n\}$, then the recursive formulas $\{\theta_n\}$ for transferring the $L_\infty$ structure on $M_L$ map to the antisymmetrization of those $\{\lambda_n\}$ for transferring the $A_\infty$ structure on $M_A$. More precisely, for any $n\geq 1$ and given $x_1,...,x_n\in M_L,$
		\begin{equation}\label{Item1}
		f\theta_n(x_1,...,x_n) = \sum_{\sigma \in S_n}\chi(\sigma) \lambda_n  \left(\widehat f (x_{\sigma(1)}) ,..., \widehat f (x_{\sigma(n)})\right).
		\end{equation} Therefore, the higher brackets are the antisymmetrization of the higher multiplications:
		\begin{equation}\label{Item2}
		\widehat f\ell_n(x_1,...,x_n) = \sum_{\sigma\in S_n}\chi(\sigma) m_n  \left( \widehat f (x_{\sigma(1)}) ,...,\widehat f (x_{\sigma(n)})\right),
		\end{equation} the terms of the induced $L_\infty$ quasi-isomorphisms $I\colon M_L\to L$ are the antisymmetrization of the terms of the $A_\infty$ quasi-isomorphism $J\colon M_A\to A$:
		\begin{equation}\label{Item3}
		fI_n(x_1,...,x_n) = \sum_{\sigma \in S_n}\chi(\sigma) J_n \left( \widehat f (x_{\sigma(1)}) ,..., \widehat f (x_{\sigma(n)})\right),
		\end{equation} and $\widehat f:M_L\to M_A$ is a strict $L_\infty$ morphism for the antisymmetrization of $\{m_n\}$.
	
\end{lemma}

\begin{remark}
	The analog of Lemma \ref{Principal} for a morphism of contractions $g:A\to L$ which is a strict $L_\infty$ morphism for the antisymmetrization of the $A_\infty$ algebra structure on $A$ also holds. 
\end{remark}

{\noindent \it Proof of Lemma \ref{Principal}:} For clarity of exposition, we prove the case in which $A=(A,d)$ is a DGA and $M_A=(HA,0)$ is its homology endowed with the trivial differential; and similarly $L=(L,\partial)$ is a DGL and $M_L=(HL,0)$. The general case follows exactly the same proof, but with more involved formulas that do not give any more insight. The multiplication map of $A$ will be denoted by $m$. We prove equation (\ref{Item1}) by induction on $n$, and deduce at each inductive step the corresponding equation for (\ref{Item2}) and for (\ref{Item3}). 

Let $n=2.$ Use, in the order given, the definition of $\theta_2$, that $f$ is a Lie map for the brackets involved, that $fi=j\widehat f$, and recognize the recursive formula for $\lambda_2:$
\begingroup
\addtolength{\jot}{1em}
\begin{align*}
f\theta_2\left(x_1,x_2\right) &= f\left[i(x_1),i(x_2)\right] = \left[fi(x_1),fi(x_2)\right] = \left[ j\widehat f(x_1),j\widehat f(x_2)\right]\\
&= m \left( j\widehat f(x_1)\otimes j\widehat f(x_2) - (-1)^{|x_1||x_2|} j\widehat f(x_2)\otimes j\widehat f(x_1)\right)\\
&= \left(m\circ j\otimes j\right) \left(\widehat f(x_1)\otimes \widehat f(x_2) - (-1)^{|x_1||x_2|} \widehat f(x_2)\otimes \widehat f(x_1)\right)\\
&= \lambda_2 \left(\widehat f(x_1)\otimes \widehat f(x_2) - (-1)^{|x_1||x_2|} \widehat f(x_2)\otimes \widehat f(x_1)\right).
\end{align*}
\endgroup Equation (\ref{Item1}) is therefore proven. Using that $f$ is a morphism of contractions, and the proof of the case $n=2$ above, we can easily prove equations (\ref{Item2}) and (\ref{Item3}):
\begin{align*}
\widehat f \ell_2\left(x_1,x_2\right) &= \widehat f q \theta_2 \left(x_1,x_2\right) = pf \theta_2 \left(x_1,x_2\right) = p\lambda_2\left(\widehat f(x_1)\otimes \widehat f(x_2) - (-1)^{|x_1||x_2|} \widehat f(x_2)\otimes \widehat f(x_1)\right)\\[0.2cm]
&= m_2\left(\widehat f(x_1)\otimes \widehat f(x_2) - (-1)^{|x_1||x_2|} \widehat f(x_2)\otimes \widehat f(x_1)\right);
\end{align*}
\begin{align*}
fI_2\left(x_1,x_2\right) &= fk \theta_2 \left(x_1,x_2\right) = Gf \theta_2 \left(x_1,x_2\right) = G\lambda_2 \left(\widehat f(x_1)\otimes \widehat f(x_2) - (-1)^{|x_1||x_2|} \widehat f(x_2)\otimes \widehat f(x_1)\right)\\[0.2cm]
&= J_2 \left(\widehat f(x_1)\otimes \widehat f(x_2) - (-1)^{|x_1||x_2|} \widehat f(x_2)\otimes \widehat f(x_1)\right).
\end{align*}

Assume next that for every $p\leq n-1$, equation (\ref{Item1}) holds. Then, (\ref{Item2}) and (\ref{Item3}) also hold for $p\leq n-1$, which follows from a manipulation identical to the one done for the case $n=2$. Let us prove that equation (\ref{Item1}) holds for $p=n$, and then also equations (\ref{Item2}) and (\ref{Item3}) for $p=n$ are straightforward consequence of $f$ being a morphism of contractions and the just proven case $n$ of equation \ref{Item1}. To lighten notation, we write $\chi(\sigma):=\varepsilon(\sigma)\sgn(\sigma)$ for any given permutation $\sigma.$

Use, in the order given: the definition of $\theta_n$, that $f$ is a Lie map for the brackets involved, the identity $fi=j\widehat f$ and the induction hypothesis, and rearrange the permutations accordingly, to end up with the recursive formula of $\lambda_n$ evaluated at the desired elements:
\begingroup
\addtolength{\jot}{1em}
\begin{align*}
f\theta_n \left(x_1,...,x_n\right) &= \sum_{s=1}^{n-1} \sum_{\sigma\in S(s,n-s)} \varepsilon(\sigma)f \left[I_s\left(x_{\sigma(1)},...,x_{\sigma(s)}\right),I_{n-s}\left(x_{\sigma(s+1)},...,x_{\sigma(n)}\right)\right]\\
&= \sum_{s=1}^{n-1} \sum_{\sigma\in S(s,n-s)} \varepsilon(\sigma) \left[fI_s\left(x_{\sigma(1)},...,x_{\sigma(s)}\right),fI_{n-s}\left(x_{\sigma(s+1)},...,x_{\sigma(n)}\right)\right]
\end{align*}
\begin{align*}
&= \sum_{s=1}^{n-1} \sum_{\sigma\in S(s,n-s)} \varepsilon(\sigma) \left[J_s\left(\sum_{\tau\in S_s}\chi(\tau) \widehat f(x_{\tau \sigma(1)})\otimes \cdots\otimes \widehat f(x_{\tau \sigma(s)})    \right),J_{n-s}\left(  \sum_{\rho\in S_{n-s}}\chi(\rho) \widehat f(x_{\rho \sigma(s+1)})\otimes \cdots \otimes\widehat f(x_{\rho \sigma(n)})  \right)\right]
\end{align*}
\begin{align*}
&= \sum_{s=1}^{n-1}\sum_{\sigma\in S(s,n-s)}\sum_{\substack{\tau\in S_s\\ \rho\in S_{n-s}}} \varepsilon(\sigma)\chi(\tau)\chi(\rho)\left[J_s\left(\widehat f(x_{\tau\sigma(1)}),...,\widehat f(x_{\tau\sigma(s)})\right),J_{n-s}\left(\widehat f(x_{\rho\sigma(s+1)}),...,\widehat f(x_{\rho\sigma(n)})\right) \right]\\
&= \sum_{s=1}^{n-1}\sum_{\sigma\in S_n}(-1)^{s+1} \chi(\sigma) \left[J_s\left(\widehat f(x_{\sigma(1)}),...,\widehat f(x_{\sigma(s)})\right),J_{n-s}\left(\widehat f(x_{\sigma(s+1)}),...,\widehat f(x_{\sigma(n)})\right) \right]
\end{align*}
\begin{align*}
&= m \Bigg( \sum_{s=1}^{n-1}\sum_{\sigma\in S_n}(-1)^{s+1} \chi(\sigma) \Big(J_s\left(\widehat f(x_{\sigma(1)}),...,\widehat f(x_{\sigma(s)})\right)\otimes J_{n-s}\left(\widehat f(x_{\sigma(s+1)}),...,\widehat f(x_{\sigma(n)})\right) \\ 
&\qquad \qquad\qquad \qquad -(-1)^{\alpha} J_{n-s}\left(\widehat f(x_{\sigma(s+1)}),...,\widehat f(x_{\sigma(n)})\right) \otimes J_s\left(\widehat f(x_{\sigma(1)}),...,\widehat f(x_{\sigma(s)})\right) \Bigg)
\end{align*}
\begin{align*}
&= \lambda_n \left(\sum_{\sigma\in S_n} \chi(\sigma) \widehat f(x_{\sigma(1)}) \otimes \cdots \otimes \widehat f(x_{\sigma(n)})\right).& &
\end{align*} 
\endgroup \hfill$\square$ \\

{\noindent \it Proof of Theorem \ref{Uno}:} To prove $(i)$, we show that fixed a contraction of $L$ onto $HL$, one can choose a contraction of $UL$ onto its homology $HUL\cong UHL \cong \Lambda H$ so that the PBW map $L\hookrightarrow UL$ is a morphism of contractions, and then apply Lemma \ref{Principal}. Let $(i,q,K)$ be a contraction of $L$ onto $H=HL$, and write $L=B\oplus \partial B \oplus C$ for the graded vector space decomposition equivalent to it. By the PBW theorem (\cite[Thm. 21.1]{Yve12}) and some basic facts of differential graded algebra, there are graded vector space isomorphisms 
\begin{equation*}
UL \cong \Lambda L \cong \Lambda\left(B\oplus \partial B \oplus C\right) \cong \Lambda \left(B\oplus \partial B\right)\otimes \Lambda C\cong \Lambda \left(B\oplus \partial B\right)\otimes UH.
\end{equation*} Since $\Lambda\left(B\oplus \partial B\right)$ is acyclic, the injection $j:\left(UH,0\right)\hookrightarrow \left(UL,d \right)$ into a quasi-isomorphism,
\begin{equation*}
\begin{tikzcd}
j:\left(UH,0\right) \ar[hook]{r}{\simeq} & \left(\Lambda\left(B\oplus \partial B\right)\otimes UH,d \right) \ar{r}{\cong} & \left(UL,d\right).
\end{tikzcd}
\end{equation*}  Decompose $UL\cong \Lambda \left(B \oplus \partial B\right) \otimes UH$, let $p:UL \to UH \cong 1\otimes UH $ be the projection onto $UH$, and let $G$ be the inverse of $d:\Lambda B \xrightarrow{\cong} \Lambda \partial B$ extended to all of $UL$ as zero in the subspace $\Lambda B\otimes 1 \otimes UH \subseteq UL.$ Then, $(j,p,G)$ is a contraction of $UL$ onto $UH$ which is a morphism of retracts for the inclusion $L=B\oplus \partial B \oplus C \hookrightarrow UL=\Lambda\left(B\oplus \partial B \oplus C\right)$. 

To prove $(ii)$, denote by $\{\mu_n\}$ the $A_\infty$ algebra structure on $UH$ induced by Baranovsky's construction, and by $\{m_n\}$ the induced by the contraction $(j,p,G)$. Since $(L,\partial)$ is a DGL, Baranovsky's construction coincides with the classical universal enveloping DGA (\cite[Thm. 3]{Bar08}). The $L_\infty$ quasi-isomorphism $Q:(L,\partial) \xrightarrow{\simeq} \left( H, \{\ell_n\}\right)$ provided by the contraction $(i,q,K)$ transforms (by \cite[Thm. 3]{Bar08}) into an $A_\infty$ algebra quasi-isomorphism $U(Q): \left(UL,d\right) \xrightarrow{\simeq} \left( UH,  \left\{ \mu_n\right\} \right)$. There is another $A_\infty$ algebra quasi-isomorphism $P:\left( UL,d\right) \xrightarrow{\simeq} \left(UH, \{m_n\}\right)$ induced by the contraction $(j,p,G)$. Hence, there is a zig-zag of $A_\infty$ quasi-isomorphisms 
\begin{equation*}
\begin{tikzcd}
\left(UH, \left\{m_n\right\}\right) 	& \left(UL,d\right) \ar[swap]{l}{\simeq} \ar{r}{\simeq}& \left(UH, \left\{\mu_n\right\}\right) 
\end{tikzcd}
\end{equation*} Since $\{m_n\}$ and $\{\mu_n\}$ are minimal, the two $A_\infty$ algebra structures are $A_\infty$-isomorphic. \hfill$\square$ \\

The results above motivate Def. \ref{DefUnive} for the universal enveloping $A_\infty$ algebra on an $L_\infty$ algebra.  Recall that any $L_\infty$ algebra $L$ is $L_\infty$ quasi-isomorphic to the DGL $\mathcal L \mathcal C (L)$ (\cite{Lod12}), and that every $L_\infty$ algebra has a minimal model (\cite[Thm. 7.9]{Mar12}). Here, $\mathcal L:\mathsf{CDGC} \leftrightarrows \mathsf{DGL}:\mathcal C$ are the adjoint functors introduced by Quillen (\cite{Qui69}), with no bounding assumptions on the underlying complexes (\cite{Hin01}).

\begin{definition}\label{DefUnive}
	Let $L$ be an $L_\infty$ algebra. Its \emph{universal enveloping $A_\infty$ algebra} is $$U_t(L) := \left(\Lambda L, \left\{m_n\right\}\right),$$ where $\{m_n\}$ is any $A_\infty$ algebra structure arising by exhibiting $\Lambda L$ as a contraction of $\Omega \mathcal C(L)$. In particular, if $L$ is minimal, then the $A_\infty$ structure on $\Lambda L$ is the one given in Theorem \ref{Uno}. 
\end{definition}

The definition given is basically equivalent to Baranovsky's. The difference is that we explicitly use Thm. \ref{Uno} for constructing it, hence avoiding the use of Baranovsky's chain homotopy $K$ \cite[Thm. 1]{Bar08}, and with explicit, more transparent formulas whenever $L$ is minimal. A different way of reading Def. \ref{DefUnive} is as follows. For an arbitrary $L_\infty$ algebra $L$, the $A_\infty$ structure $\{m_n\}$ on $\Lambda L$ arises by forming the diagram:
\begin{equation*}
\begin{tikzcd}
\Omega \mathcal C\left(L\right) \ar[loop left]{l}{} \ar[shift left]{r}{} 	& \Lambda L \ar[shift left]{l}{}\\[0.35cm]
\mathcal L \mathcal C \left(L\right) \ar[loop left]{l}{} \ar[shift left]{r}{} \ar{u}{}		& L \ar[shift left]{l}{} \ar[,swap]{u}{}
\end{tikzcd}
\end{equation*} From this point of view, we start with a contraction from $\mathcal L \mathcal C \left(L\right)$ onto $L$ producing the $L_\infty$ structure of $L$, and then the proof of Theorem \ref{Uno} goes through: the classical PBW map $$\mathcal L \mathcal C \left(L\right) \hookrightarrow U\left(\mathcal L \mathcal C \left(L\right)\right) =\Omega \mathcal C\left(L\right)$$ is made a morphism of contractions, where we contract $\Omega \mathcal C\left(L\right)$ onto its homology $H_*\left(\Omega \mathcal C\left(L\right)\right)$, which is isomorphic as a graded vector space to $\Lambda L$ (this isomorphism follows, for example, from \cite[Thm. 1]{Bar08}). Given $f:L_1 \to L_2$ an $L_\infty$ morphism, and once chosen contractions 
\begin{equation*}
\begin{tikzcd}
\Omega \mathcal C(L_i) \ar[loop left]{l}{} \ar[shift left]{r}{p_i} & \Lambda L_i = U_t\left(L_i\right), \ar[shift left]{l}{j_i} \quad i=1,2,
\end{tikzcd}
\end{equation*} there is a uniquely defined $A_\infty$ morphism $$U_t(f)=p_2\circ \Omega \mathcal C(f) \circ j_1 :U_t(L_1) \to U_t(L_2),$$ enjoying  properties similar to Baranovsky's definition on morphisms (see \cite[Thm. 3]{Bar08}).

\section{Homotopical properties and comparison with other envelopes}\label{comparacion}

We collect  the main properties regarding the homotopy type of the several universal enveloping constructions in Proposition \ref{SonHomotopas}. \\

Let $L$ be an $L_\infty$ algebra. Denote by $U_B(L)$ and $U_t(L)$ the construction of Baranovsky and the given in Def. \ref{DefUnive}, respectively. We will consider a third universal $A_\infty$ envelope $U_d(L)$, see discussion after Conjecture \ref{DGCML}. At this point, it suffices to know that $U_d(L)$ is isomorphic to  $\Lambda L$ as a graded vector space, and  carries an $A_\infty$ structure for which there is a DGC quasi-isomorphism $$\mathcal C(L) \xrightarrow{\simeq} BU_d(L).$$ The universal envelopes $U_B,U_t$ and $U_d$ are homotopy equivalent (Prop. \ref{SonHomotopas} \textit{$(i)$}).  Quillen's foundation of rational homotopy theory, as well as other deep results (see for example \cite{Ani89,Hal92,Maj00}), heavily rely on the now classical fact that homology commutes with the classical universal enveloping algebra functor over characteristic zero fields, 
\begin{equation}\label{Conmuta}
UH=HU.
\end{equation} See \cite[Appendix B]{Qui69}. The identity (\ref{Conmuta}) holds only up to homotopy for the universal enveloping constructions $U_B, U_t, U_d$ and $\mathcal U$ (Prop. \ref{SonHomotopas} $(iii)$), where $\mathcal U$ is Lada and Markl's universal enveloping (\cite{Lad95}). Another classical result of Quillen (\cite{Qui69}, see also \cite{Nei78A}) asserts that for a given DGL $L$ with universal enveloping DGA $UL$, there is a natural DGC quasi-isomorphism 
\begin{equation}\label{DGCQuasiQuillen}
\mathcal C(L) \xrightarrow{\simeq} BUL.
\end{equation}  For $L_\infty$ algebras, although $\mathcal C(L)$, $BU_t(L)$ and $BU_B(L)$ are DGC's, there is usually no direct DGC quasi-isomorphism as in (\ref{DGCQuasiQuillen}). However, these DGC's are always weakly equivalent, which is the lift of the quasi-isomorphism (\ref{DGCQuasiQuillen}) when dealing with infinity structures (Prop. \ref{SonHomotopas} $(ii)$).

\begin{proposition}\label{SonHomotopas} Let $L$ be an  $L_\infty$ algebra. Then,
	\begin{enumerate}
		\item[(i)] There are $A_\infty$ quasi-isomorphisms $$U_t(L)\simeq U_B(L)\simeq U_d(L).$$ The three constructions are then the same up to homotopy, and $U_t(L)\cong U_B(L)$ are isomorphic if $L$ is minimal.
		\item[(ii)] There is an $A_\infty$ coalgebra quasi-isomorphism $$\mathcal C(L) \xrightarrow{\simeq} BUL,$$ where $U$ is any of the envelopes $U_t,U_B$ or $U_d$, which is not generally a DGC map for $U_B$ or $U_t$.
		\item[(iii)]  Assume that $H_*\left(L\right)$ carries an $L_\infty$ structure induced by a contraction from $L$ onto it. Then, there are $A_\infty$ quasi-isomorphisms $$U \left(H_*\left(L\right)\right) \simeq H_*\left(UL\right),$$ where $U$ is any of the envelopes $U_t,U_B,U_d$ or $\mathcal U$.
	\end{enumerate}
\end{proposition}

\begin{proof} $(i)$  Theorem \ref{Uno}$(ii)$ asserts that $U_B(L)\simeq U_t(L).$ It suffices to show that $U_B(L)\simeq U_d(L).$ Indeed, the construction of $U_d$ is based on the existence of a differential $d$ on $T\left(s\Lambda^+ L\right)=BU_d(L)$ so that there is a DGC quasi-isomorphism $\mathcal C (L)\xrightarrow{\simeq }BU_d(L).$ By \cite[Thm 4 (ii)]{Bar08}, there is a DGA quasi-isomorphism $\Omega \mathcal C(L)\to \Omega BU_B(L)$. Since the bar construction preserves quasi-isomorphisms, and given that the unit of the bar-cobar adjunction is a quasi-isomorphism for conilpotent coalgebras, there is the following zig-zag of DGC quasi-isomorphisms, from which the result follows: 
	\begin{equation}
		\begin{tikzcd}
			BU_d(L) & \mathcal C(L) \arrow[l] \arrow[r] & B\Omega \mathcal C(L) \arrow[r] & B \Omega B U_B(L) & BU_B(L) \arrow[l]
		\end{tikzcd}
	\end{equation} $(ii)$ Follows from the zig-zag just above.\\ $(iii)$ 	By item $(i)$, it suffices to prove it for $U=U_B$ and for $U=\mathcal U$. Let $f:L\to HL$ be an $L_\infty$ quasi-isomorphism. Since $U_B$ preserves quasi-isomorphisms, $U_B(f):U_B(L)\to U_B(HL)$ is an $A_\infty$ quasi-isomorphism. Thm. \ref{HTT} provides an $A_\infty$ algebra structure on $H(U_B(L))$, as well as an $A_\infty$ quasi-isomorphism $I:H(U_B(L)) \to U_B(L)$. Thus, the following composition is an $A_\infty$ quasi-isomorphism: $$H(U_B(L)) \xrightarrow{I} U_B(L) \xrightarrow{U_B(f)} U_B(HL).$$ Let us prove it for $\mathcal U$. Fix a contraction 
	\begin{equation}\label{Con1}
	\begin{tikzcd}
	L \ar[loop left]{l}{K} \ar[shift left]{r}{q} & H, \ar[shift left]{l}{i}
	\end{tikzcd}
	\end{equation} endow $H$ with an $L_\infty$ structure via Thm. \ref{HTT}, and denote by $\{m_n\}$ the $A_\infty$ structure on $\mathcal UL.$ Markl's PBW-infinity theorem \cite[Thm. 4.7]{Mar05} gives an isomorphism of $A_\infty$ algebras $$S^*\left(L\right) \xrightarrow{\cong} G^*\left(L\right).$$ Here, $G^*\left(L\right)$ is the associated graded $A_\infty$ algebra for the ascending filtration of $\mathcal UL$ given by $F_0=\Q, F_1=\Q\oplus L$, and for $p\geq 2:$
	\begin{equation*}
	F_pL=\operatorname{Span}_{\Q} \left\{m_n\left(x_1,.,,,.x_n\right) \mid n\geq 2, \ x_j\in F_{p_j}L, \ p_1 + \cdots + p_n \leq p\right\},
	\end{equation*} and $$S^*\left(L\right)=\mathcal F\left(L,\ell_1\right)/J$$ is the quotient of the free $A_\infty$ algebra on the chain complex $\left(L,\ell_1\right)$ by the ideal generated by imposing the vanishing on $L$ of the antisymmetrization of the $A_\infty$ structure $\left\{\mu_n\right\}$ of $\mathcal F\left(L,\ell_1\right)$ for $n\geq 2$. That is, $$\mu_n^\mathcal L\left(x_1,...,x_n\right)=0 \quad \textrm{for all } n\geq 2, \ x_i\in L.$$ Basically, $S^*$ is the "free $A_\infty$ algebra symmetrized on $L$" (not to be confused with a $C_\infty$ algebra, whose structural maps vanish on the image of the shuffle products). Denote by $\mathcal P$ the dg operad whose free algebras are given by $S^*$ (an explicit description in terms of planar trees is given in \cite[Prop. 4.6]{Mar05}). Summarizing, for any $L_\infty$ algebra $L$, there is an isomorphism of $A_\infty$ algebras $$UL\cong S^*\left(L\right),$$ where $S^*\left(L\right)=\mathcal P\left(L\right)$ is the free $\mathcal P$-algebra for a certain dg operad $\mathcal P.$ Thus, after a possible change of homotopy in the contraction from $L$ onto $H$, Berglund's generalization of the tensor trick to algebras over operads (\cite[Thm. 1.2]{Ber14A}) applies to the contraction (\ref{Con1}). That is, there is a contraction
	\begin{equation*}
	\begin{tikzcd}
	\mathcal UL\cong S^*\left(L\right) \ar[loop left]{l}{S^*\left(K\right)} \ar[shift left]{r}{S^*\left(q\right)} & S^*\left(HL\right)\cong \mathcal U HL. \ar[shift left]{l}{S^*\left(i\right)}
	\end{tikzcd}
	\end{equation*} To finish, choose any $A_\infty$ quasi-isomorphism $\mathcal UL\simeq H_*\left(\mathcal UL\right)$, for instance by using Thm. \ref{HTT}. Then, there are $A_\infty$ quasi-isomorphisms $$\mathcal UH_*\left(L\right) \xrightarrow{\simeq } \mathcal UL \xrightarrow{\simeq} H_*\left(\mathcal UL\right).$$	
\end{proof} 

\begin{remark}
	One could try to adapt Quillen's proof for DGL's in \cite[App. B]{Qui69} of the identity $HU=UH$ for $\mathcal U$. Several subtleties arise this way, and in fact, one \emph{cannot} improve Prop. \ref{SonHomotopas}  \textit{(iii)}. Indeed, any "natural" map $\mathcal U\left(HL\right) \to H\mathcal UL$ passes through a previous choice of infinity structures, thus one cannot expect an isomorphism. It gets even worst than that: no choice will ever be an isomorphism, except for the trivial case, given that by definition $\mathcal UHL$ carries a non-trivial differential, whereas $H_*\left(\mathcal UL\right)$ does not.
\end{remark}

For $\mathcal P$ a dg operad, recall that a $\mathcal P$-algebra is \emph{formal} if there exists a zig-zag of $\mathcal P$-algebra quasi-isomorphisms connecting it to its homology (\cite{Lod12}). In presence of a contraction, Lemma \ref{Principal} gives a straightforward proof of the fact that $L$ is formal as a DGL if, and only if, $UL$ is formal as a DGA. This result (\cite{Sal17}), however, has been superseded by \cite[Thm. B]{Cam19}.

We conclude this section with a conjecture. 

\begin{conjecture}\label{DGCML}
	Let $L$ be an $L_\infty$ algebra. Lada and Markl's universal enveloping $A_\infty$ algebra $\mathcal UL$ is such that there is a natural DGC quasi-isomorphism $\mathcal C(L) \xrightarrow{\simeq} B\mathcal UL.$
\end{conjecture}

If Conjecture \ref{DGCML} is true, then the universal enveloping $U_d$ studied in this section enjoys the homotopical properties of $\mathcal U$. This justifies the study of $U_d$. To finish the homotopical study of $\mathcal U$, it suffices to prove the weaker version of Conjecture \ref{DGCML} relaxing the DGC quasi-isomorphism to a zig-zag of quasi-isomorphisms.

\section{The Milnor-Moore infinity theorem and a new rational model} \label{CapiRat}

The algebraic formalism of Section \ref{Seccion2} has interesting applications to rational homotopy theory. The monograph \cite{Yve12} is an excellent resource on rational homotopy theory. In this section, all $L_\infty$ algebras are concentrated in non-negative degrees.

\subsection{The Milnor-Moore infinity theorem} \label{SeccionMM}
Let $X$ be a simply connected complex. The classical Milnor-Moore theorem (\cite{Mil65}) asserts that the rational homotopy Lie algebra $L_X=\pi_*\left(\Omega X\right) \otimes \Q$ embeds as the primitive elements of the rational loop space Hopf algebra $H_*(\Omega X;\Q)$. Furthermore, the latter Hopf algebra is precisely the universal enveloping algebra of $L_X$, and the inclusion is given by the rationalization of the Hurewicz morphism,
\begin{equation}\label{Hur}
h : \pi_*\left(\Omega X\right) \otimes \Q \hookrightarrow H_*(\Omega X;\Q)=U\left(\pi_*\left(\Omega X\right) \otimes \Q\right).
\end{equation} If only the rational homotopy Lie algebra $\pi_*\left(\Omega X\right) \otimes \Q$ is taken into account, then non-equivalent rational spaces may share this invariant. For instance, the rationalization of $\C P^2$ and of $K(\Z,2)\times K(\Z,5)$ are not equivalent, yet both have abelian two dimensional isomorphic rational homotopy Lie algebras. However, endowing an $L_\infty$ structure to $\pi_*\left(\Omega X\right) \otimes \Q$  determines a unique rational homotopy type, even if we include the class of nilpotent finite type complexes.  In this latter case, we need to restrict to finite type pronilpotent $L_\infty$ algebras. The rational homotopy type encoded by such an $L_\infty$ algebra $L$ is determined by the DGL $\mathcal L \mathcal C (L)$ in case $L=L_{\geq1}$, and by the Sullivan algebra $\mathcal C^*(L)$ in case $L=L_{\geq 0}$ is finite type pronilpotent. Here, $\mathcal C^* = \vee \circ \mathcal C$ is the linear dual $\vee$ of the Quillen chains $\mathcal C$. See \cite[Thm. 2.3]{Ber15} for details. By a beautiful result of Majewski, whenever $X$ is simply connected of finite type, these two models are homotopy equivalent (\cite{Maj00}).\\

Denote $U=U_t$. The next result lifts the morphism (\ref{Hur}) to the context of infinity algebras.

\begin{theorem}\label{MMInfinito} Let $X$ be a simply connected complex. Endow $\pi_*\left(\Omega X\right)\otimes \Q$ with an $L_\infty$ structure $\{\ell_n\}$ representing the rational homotopy type of $X$ for which $\ell_1=0$ and $\ell_2=[-,-]$ is the Samelson bracket. Then, there is an $A_\infty$ algebra structure $\{m_n\}$ on the loop space homology algebra $H_*\left(\Omega X;\Q\right)$ for which $m_1=0, m_2$ is the Pontryagin product, and such that the rational Hurewicz morphism 
\begin{equation*}
h : \pi_*\left(\Omega X\right) \otimes \Q \hookrightarrow H_*(\Omega X;\Q)=U\left(\pi_*\left(\Omega X\right) \otimes \Q\right)
\end{equation*} is a strict $L_\infty$ embedding. Therefore, the $L_\infty$ structure on the rational homotopy Lie algebra is the antisymmetrized of the $A_\infty$ structure on $H_*(\Omega X;\Q)$: $$\ell_n(x_1,...,x_n)=\sum_{\sigma\in S_n} \chi(\sigma) m_n\left(x_{\sigma(1)},...,x_{\sigma(n)}\right).$$
\end{theorem}

\begin{proof} Assume that the rational homotopy Lie algebra $\pi_*\left(\Omega X\right)\otimes \Q$ carries a minimal $L_\infty$ structure $\{\ell_n\}$ governing the rational homotopy type of $X$ for which $\ell_2$ is the Samelson bracket. For instance, from a CW-decomposition $$* =X^{(1)} \subseteq X^{(2)} \subseteq \cdots \subseteq \bigcup_n X^{(n)}=X,$$ build the Quillen minimal model $L=\left(\Lie(V),\partial\right)$ of $X$, satisfying $$H_*\left(L\right) \cong \pi_*\left(\Omega X\right)\otimes \Q$$ as graded Lie algebras. The choice of a contraction from $L$ onto $\pi_*\left(\Omega X\right)\otimes \Q$ gives an $L_\infty$ structure as in the statement. The rational Hurewicz homomorphism of equation (\ref{Hur}) is, after the choice of an ordered basis of $L$, the PBW map from $L$ into $UL$. Therefore, $h$ can be chosen to be $h=\widehat \imath = p \imath i$ in the following diagram, which is under the hypotheses of Theorem \ref{Principal}:
\begin{equation*}
\begin{tikzcd}
TV=U\left(\Lie(V)\right) \ar[loop left]{l}{G} 	 \ar[shift left]{r}{p}	& H_*\left(\Omega X;\Q\right)  \ar[shift left]{l}{j}\\[0.35cm]
\Lie(V)	\ar[loop left]{l}{K} 	\ar[shift left]{r}{q} \ar[hook]{u}{\imath}			& \pi_*\left(\Omega X\right)\otimes \Q  \ar[shift left]{l}{i} \ar[hook,swap]{u}{h}
\end{tikzcd}
\end{equation*} An application of Theorem \ref{Principal} finishes the proof.
\end{proof}

\begin{remark}\label{Primitivos}
	Let $U_t(L)=(\Lambda L, \{m_n\})$ be the universal enveloping $A_\infty$ algebra of $(L,\{\ell_n\})$. For each $n$, the composition 
	\begin{equation*}
	\begin{tikzcd}
	L^{\otimes n} \arrow[r, "i_n", hook] & \left(\Lambda L\right)^{\otimes n} \arrow[r, "m_n^\mathcal L"] & \Lambda L
	\end{tikzcd}
	\end{equation*} has its image in $L\subseteq \Lambda L$. Let $\pi:\Lambda L \to L$ be the projection. The primitives of $\Lambda L$ for the  standard coproduct are precisely $\mathcal P_*(\Lambda L)=L.$ Thus, the original $L_\infty$ structure can be recovered by performing two natural operations to $U_t(L)$: antisymmetrizaton and restriction to primitives.
	\begin{equation*}
	\left(\Lambda L, \{m_n\}\right) \  \longmapsto \ \left( \mathcal P_*(\Lambda L), \ \pi \circ m_n^{ \mathcal L} \circ i_n\right) = (L,\{\ell_n\}).
	\end{equation*}
\end{remark}

Detecting when a given cocommutative Hopf algebra is the universal envelope of its primitives is a difficult problem. This has been studied, among others,  by Anick, Cartier, Halperin, Kostant, Milnor and Moore. See for example \cite{Hal92}. The classical name of this sort of result is the \emph{Cartier-Milnor-Moore theorem}. Does a similar statement hold in the infinity setting? 

\begin{conjecture}
	Let $A$ be an $A_\infty$ algebra over a characteristic zero field such that there is a cocommutative, conilpotent coproduct $\Delta$ on $A$ which is a strict $A_\infty$ morphism $A\to A^{\otimes 2}$. Then, the primitives for the coproduct $L=\Ker(\overline \Delta) = \mathcal P_*(A)$ form an $L_\infty$ algebra, and the inclusion $L\hookrightarrow A$ extends to an isomorphism of $A_\infty$ algebras $$UL\xrightarrow{\cong} A$$ which respects the Hopf structure.
\end{conjecture}

 In the conjecture above, we expect $U$ to be Lada and Markl's envelope, and maybe the diagonal $\Delta$ needs to come from a "Hopf algebra up to homotopy", so that the isomorphism might be not only of $A_\infty$ algebras, but of homotopy Hopf algebras. 
 If $X$ is a simply connected complex, and $H_*(\Omega X;\Q)$ carries a universal enveloping $A_\infty$ structure, then $H_*(\Omega X;\Q)$ is a rational model for $X$. Indeed, by Remark \ref{Primitivos}, $$\mathcal P_*\left(H_*(\Omega X;\Q)\right) = \pi_*\left(\Omega X\right)\otimes \Q$$ is a fully-fledged $L_\infty$ algebra capturing the rational homotopy type of $X$.

\subsection{Examples. Recovering the Sullivan and Quillen models}\label{Ejemplos}

We explicitly record several examples of universal enveloping $A_\infty$ algebras of the sort  $$U_t\left(\pi_*\left(\Omega X\right)\otimes \Q , \left\{\ell_n\right\}\right)=\left(H_*(\Omega X; \Q), \{m_n\}\right).$$

\begin{enumerate}
\item \textbf{The simply connected sphere $S^n$}.
\begin{itemize}
\item For odd $n$, it is $\Lambda x$ with $|x|=n-1$, with trivial differential and trivial higher multiplications of all orders. 
\item For even $n$, it is $\Lambda (x,y)$ with $|x|=n-1$, $|y|=2n-2$, with a unique non-trivial multiplication map given by $m_2(x,x)= \frac{1}{2}y.$
\end{itemize}
\item \textbf{A finite product of simply-connected Eilenberg-Mac Lane spaces  $\prod_{i=1}^k K(\Q,n_i)$}.\\ It is given by \begin{center} $\left(\Lambda x_1,...,x_k\right)$, where each $|x_i|= n_i-1$, \end{center} with trivial differential and higher multiplications of all orders. 
\item \textbf{The complex projective spaces $\C P^k$, for $k\geq 1$.} \\ It is given by $\left( \Lambda x,y\right)$, with $|x|=1$, $|y|=2k$ and its only non-trivial higher multiplication is $$m_{k+1}(x,...,x)=\frac{1}{(k+1)!^2}y.$$ Indeed, an $L_\infty$ model $L=\pi_*(\Omega \C P^k)\otimes \Q$ of $\C P^k$ has a basis $\{x,y\}$ with $|x|=1, |y|= 2k$ with a single non-vanishing higher bracket, given by $\ell_{k+1}(x,...,x)=\frac{1}{(k+1)!}y.$ The result then follows, since the sign $\chi(\sigma)$ in the sum below is always positive:
\begin{equation*}
\frac{1}{(k+1)!}y =  \ell_{k+1}(x,...,x) = \sum_{\sigma \in S_{k+1}} \chi(\sigma) m_{k+1}(x,...,x)  = (k+1)!m_{k+1}(x,...,x).
\end{equation*}
\item \textbf{Coformal spaces}. \\ The universal enveloping $A_\infty$ algebra model of any coformal space can be chosen to be the classical universal enveloping algebra of it. Indeed, if $X$ is coformal,  then $L=\pi_*(\Omega X)\otimes \Q$ together with $\ell_2$ given by the Samelson product is an $L_\infty$ model of $X$. Since $L$ is a DGL with trivial differential, the universal enveloping $A_\infty$ algebra of it coincides with the classical envelope, having the latter trivial differential as well. This includes examples 1 and 2.
\end{enumerate}

Let $U_t(L)= \left( \Lambda L, \{m_n\}\right)$ be universal enveloping $A_\infty$ model of a simply connected complex $X$. Let  $L=\mathcal P_ *\left(H_*\left(\Omega X;\Q\right)\right)$ be the primitives for the natural diagonal (Rmk. \ref{Primitivos}). Then, one recovers:
\begin{itemize}
\item Provided $X$ is of finite type, a (not necessarily minimal) Sullivan model $\left(\Lambda V,d\right)$ of $X$ by setting $V=\left(sL \right)^{\vee}$ and $d=\sum_{n \geq 1}d_n$ determined by the pairing
\begin{equation}\label{Pairing}
\left \langle d_n(v), sx_1 \wedge ... \wedge sx_n \right\rangle = \varepsilon \sum_{\sigma \in S_n} \chi(\sigma) \left\langle v; sm_n\left( x_{\sigma(1)},...,x_{\sigma(n)} \right) \right \rangle,
\end{equation} where $\varepsilon$ is the parity of $\sum_{j=1}^{n-1}(n-j)|x_j|.$ 
\item A (not necessarily minimal) Quillen model by setting 
\begin{equation*}
\left( \Lie(U),\partial \right) = \left( \Lie\left(s^{-1}\Lambda^+ sL\right),\partial_1+\partial_2 \right) = \mathcal L \mathcal C \left(\mathcal P_ *\left(H_*\left(\Omega X;\Q\right)\right), \{m_n^{\mathcal L}\}\right).
\end{equation*} The quadratic part $\partial_2$ of the differential is the standard induced by the reduced  coproduct of $\mathcal{C}(L)$ (see formula (\ref{FormulaPartial})), and $\partial_1$ is explicitly given on generators by
\begin{align*}
\partial_1 \big(s^{-1} \big( sx_1 \wedge &... \wedge sx_p \big) \big)  \\& = \ \sum_{k=1}^p \sum_{i_1 \leq \cdots \leq i_k} \sum_{\sigma \in S_k} \varepsilon^\sigma_{(i_1,...,i_k)} s^{-1} \left(sm_k\left(x_{i_{\sigma(1)}},...,x_{i_{\sigma(k)}}\right) \wedge sx_1 ... \widehat{sx_{i_1}}...\widehat{sx_{i_k}}... \wedge sx_{i_p}\right).
\end{align*} The sign $\varepsilon^\sigma_{(i_1,...,i_k)}=-\varepsilon (-1)^{n_{i_1...i_k}}\chi(\sigma)$ is given by the Koszul sign rule and the involved maps. 
\end{itemize}

\subsection{Higher Whitehead products and Pontryagin-Massey products}\label{SeccionHigherPontryagin}

Several authors have related the (ordinary, as well as higher) Whitehead products $[-,-]$ on $\pi_*(X)$ with the Pontryagin product $*$ on $H_*\left(\Omega X;R\right)$. For instance, the main result in  \cite{Sam53} states that the two-fold Whitehead product  of $x\in \pi_{n+1}$ and $y\in \pi_{m+1}$  is an antisymmetrized Pontryagin product: 
\begin{equation*}
h[x,y] =(-1)^{n} \left(h(x)*h(y)-(-1)^{nm}h(y)*h(x)\right).
\end{equation*} Here, $h:\pi_*(X)\xrightarrow{\cong} \pi_{*-1}(\Omega X) \to H_{*-1}(\Omega X;\Z)$ is the Hurewicz morphism precomposed with an isomorphism. In \cite[Thm 3.3]{Ark71}, it is shown that under some hypothesis, certain higher order Whitehead product sets $[x_1,...,x_k]_W\subseteq \pi_{*}(X)$ are non-empty, and contain an element which is a sort of generalized $k$-fold Pontryagin product. \\

In the rational case, Thm. \ref{MMInfinito} is the most general form of these sort of statements. Assuming the existence of non-trivial higher products in a sense to be explained, one can go a step further and extract an interesting relationship. For space considerations, and since this section is about an application of the main results of this work, we omit a (necessarily lengthy) explanation of the necessary background. Instead, we refer the reader to \cite{Tan83} for background on the (rational) higher order Whitehead products, and to \cite{HigherWhitehead} for an account of their relationship with $L_\infty$ structures. We start with the following observation.

\begin{proposition}
	Let $X$ be a simply connected complex. The $A_\infty$ algebra structures on $H_*(\Omega X;\Q)$ arising from exhibiting $H_*(\Omega X;\Q)$ as a contraction of the chains DGA $C_*(\Omega X;\Q)$ and by taking universal enveloping $A_\infty$ algebra of an $L_\infty$ model on $\pi_*(\Omega X)\otimes \Q$ are $A_\infty$ quasi-isomorphic.
\end{proposition}

\begin{proof}
	Let $L= \left(\pi_*(\Omega X)\otimes \Q, \{\ell_n\}\right)$ be the $L_\infty$ model of $X$, and assume without loss of generality that $L$ arises as a contraction of the Quillen model $(\Lie(U),\partial)$ of $X$. Denote by $\{m_n\}$ the $A_\infty$ structure on $H_*(\Omega X;\Q)$ arising from Thm. \ref{Uno}. There is a square 
	
	\begin{equation*}
	\begin{tikzcd}
	U\left(\Lie(V)\right)  	 \ar[shift left]{r}{\simeq}	& H_*\left(\Omega X;\Q\right) \\[0.35cm]
	\Lie(V)		\ar[shift left]{r}{\simeq} \ar[hook]{u}{}			& \pi_*\left(\Omega X\right)\otimes \Q   \ar[hook,swap]{u}{}
	\end{tikzcd}
	\end{equation*}
	whose horizontal top and bottom arrows are $A_\infty$ and $L_\infty$ quasi-isomorphisms, respectively. Since there is a DGL quasi-isomorphism $\Lie(U) \xrightarrow{\simeq} \lambda(X)$ onto the Quillen construction $\lambda(X)$ (\cite{Qui69}), and the classical enveloping functor $U$ preserves quasi-isomorphisms (\cite[Thm. 21.7]{Yve12}), there is a DGA quasi-isomorphism $U\Lie(U) \xrightarrow{\simeq} U\lambda(X)$. Since $U\lambda(X)$ is weakly equivalent to $C_*(\Omega X;\Q)$ as a DGA, there is an $A_\infty$ quasi-isomorphism $U\lambda(X) \xrightarrow{\simeq} (H_*(\Omega X;\Q), \{m_n'\})$ for $\{m_n'\}$ induced by exhibiting $H_*(\Omega X;\Q)$ as a contraction of $C_*(\Omega X;\Q)$.
\end{proof}

The Massey products of a space $X$ are certain higher order operations on the cohomology algebra $H^*(X;R)$. These arise from relations between the cup product and the differential in the singular cochains $C^*(X;R)$, see  \cite{Mas58,May69}. The Massey products and the $A_\infty$ structures on $H^*(X;R)$ are tightly related, see \cite{Bui18} for details. Both, the Massey products and $A_\infty$ structure, exist in the homology $H$ of any DGA $A$ - one needs not consider these operations only when $A$ is the singular cochain algebra of a space. So, given that $H_*(\Omega X;R)$ is the homology of the DGA $C_*(\Omega X;R)$ for the Pontryagin product, it makes sense to consider the algebraic Massey products on $H_*(\Omega X;R)$. We call these higher products on $H_*(\Omega X;R)$ arising from relations between the Pontryagin product and the differential of the DGA $C_*(\Omega X;R)$ the \emph{higher Massey-Pontryagin products} of $X$. This way, we avoid the confusion with the classical Massey products of $X$. Again for space considerations, we refer the reader to the   works mentioned in this paragraph for the necessary background on Massey products and $A_\infty$ structures. \\

Denote by $h:\pi_*\left( \Omega X\right)\otimes \Q \to H_*\left(\Omega X;\Q\right)$ the rational Hurewicz morphism. 

\begin{theorem}\label{HigherWhiteheadHigherPontryagin} Let $x_1,...,x_n \in \pi_*\left( \Omega X\right)\otimes \Q$, and denote by $y_k=h\left(x_k\right)\in H_*\left(\Omega X;\Q\right)$ the corresponding spherical classes. Assume that the higher Whitehead product set $\left[x_1,...,x_n\right]_W$ and the higher Massey-Pontryagin products sets $\left\langle y_{\sigma(1)},...,y_{\sigma(n)} \right\rangle$ for every $\sigma \in S_n$ are defined. If the $A_\infty$ algebra structure $\{m_k\}$ on $H_*\left(\Omega X;\Q\right)$ provided by Thm. \ref{MMInfinito} has vanishing $m_k$ for $k\leq n-2$, then $x=\varepsilon\ell_n\left(x_1,...,x_n\right)\in \left[x_1,...,x_n\right]_W,$ and satisfies:
\begin{equation*}
h(x) \in \sum_{\sigma \in S_n} \chi(\sigma) \left\langle y_{\sigma(1)},...,y_{\sigma(n)}\right\rangle.
\end{equation*} Here, $\varepsilon$ is the parity of $\sum_{j=1}^{n-1} |x_j|(k-j)$. If moreover the involved higher products are all uniquely defined, then the above containment is an equality of elements.
\end{theorem}

Since the particular case $n=3$ of the result above is the most likely to be computed, and in this case the hypothesis that $m_1=0$ is superfluous, we consider this case to be of independent interest.

\begin{corollary}\label{CasoDe3} Let $x_1,x_2,x_3 \in \pi_*\left( \Omega X\right)\otimes \Q$, and denote by $y_k=h\left(x_k\right)\in H_*\left(\Omega X;\Q\right)$ the corresponding spherical classes. Assume that the triple Whitehead product $\left[ x_1,x_2,x_3\right]_W$ and the triple Massey products $\left\langle y_{\sigma(1)},y_{\sigma(2)},y_{\sigma(3)} \right\rangle, \sigma \in S_3,$ are defined. Then $x=\varepsilon \ell_3\left(x_1,x_2,x_3\right) \in \left[ x_1,x_2,x_3\right]_W$, and satisfies:
\begin{equation*}
h(x) \in \sum_{\sigma \in S_3} \chi(\sigma) \left\langle y_{\sigma(1)},y_{\sigma(2)},y_{\sigma(3)}\right\rangle.
\end{equation*}If moreover the triple products are all uniquely defined, then the above containment is an equality of elements.
\end{corollary}

{\noindent \it Proof of Theorem \ref{HigherWhiteheadHigherPontryagin}:} Since $m_k=0$ for every $k\leq n-2$, it follows from Thm. \ref{MMInfinito} that also $\ell_k=0$ vanishes whenever $k\leq n-2.$ Therefore, \cite[Thm. 3.5]{HigherWhitehead} asserts that $x=\varepsilon\ell_n\left(x_1,...,x_n\right)\in \left[x_1,...,x_n\right],$ meanwhile its associative counterpart \cite[Thm 3.3]{Bui18} asserts that $\varepsilon_\sigma m_n \left(y_{\sigma(1)},...,y_{\sigma(n)}\right) \in \langle y_{\sigma(1)},...,y_{\sigma(n)}\rangle$. We are denoting by $\varepsilon_\sigma$ the parity of $\sum_{j=1}^{n-1}(k-j)|x_{\sigma(j)}|$. Using Thm. \ref{MMInfinito}, we conclude that:
\begin{align*}
h(x) &= \varepsilon h\ell_n(x_1,...,x_n) = \varepsilon \left( \sum_{\sigma\in S_n} \chi(\sigma)\varepsilon_\sigma m_n \left(y_{\sigma(1)},...,y_{\sigma(n)}\right) \right)\\
&=\sum_{\sigma\in S_n} \chi(\sigma) m_n \left(y_{\sigma(1)},...,y_{\sigma(n)}\right) \in \sum_{\sigma\in S_n} \chi(\sigma) \left\langle y_{\sigma(1)},...,y_{\sigma(n)}\right\rangle. 
\end{align*} \hfill$\square$

\bibliographystyle{plain}
\bibliography{MyBib}

\noindent\sc{José M. Moreno-Fernández}\\ 
\noindent\sc{Max Planck Institute for Mathematics \\ Vivatsgasse 7, 53111 Bonn, Germany}\\
\noindent\tt{josemoreno@mpim-bonn.mpg.de}

\end{document}